\documentclass{article}
\usepackage{amsmath,amssymb, epsf}
\usepackage{theorem}
 
\newtheorem{thm}{Theorem}

\newtheorem{lemma}{Lemma}
{\theorembodyfont{\rmfamily} \newtheorem{example}{Example}}
\newenvironment{proof}{\noindent\textsc{Proof}}{\hfill\ensuremath{\blacksquare}}

\newenvironment{ack}{\noindent\textbf{Acknowledgements}}{}

\newcommand{\cK}{{\cal K}}
\newcommand{\cJ}{{\cal J}}
\newcommand{\cE}{{\cal E}}
\newcommand{\cW}{{\cal W}}
\begin{document}

\title{A Proof of Solomon's Rule} 
\author{Stephanie J. van Willigenburg\\Department of Mathematics and 
Statistics,\\York University, 4700 Keele St, \\North
York, ON, M3J 1P3, CANADA.} 
\maketitle 
\begin{abstract}We put forward a proof   of Solomon's rule, in terms of matrices, for 
multiplication
in the descent algebra of the symmetric group. Our proof exploits the
graphs that we  can  obtain from all the subsets of the set of transpositions,
$\{(i,i+1)\}_{i=1}^{n-1}$. 
\end{abstract}

Let $W$ be a Coxeter group with generating set, $S$, of fundamental reflections. If $J$ is any 
subset of $S$, let
$W_J$ be the subgroup generated by $J$. Let $X_J$ be the unique set of minimal length left coset
representatives of $W_J$. Note that $X_J^{-1}=\{x^{-1}|x\in X_J\}$    is the unique set of minimal
length representatives for the right cosets of
 $W_J$.  Let $l(y)$ denote the length of $y$ in $W$.

Solomon \cite{solomon-mackey} then gives us the following theorem:
\begin{thm}\label{solomon-theorem}
For every subset $K$ of $S$, let
\[{\cal X}_K=\sum_{\sigma\in X_K}\sigma.\]
Then for subsets $J$ and $K$ in $S$
\[{\cal X}_J{\cal X}_K=\sum  _{x\in X_J ^{-1}\cap X_K}{\cal X}_{x^{-1}Jx\cap K}\]
\end{thm}

From this it follows that the set of all ${\cal X}_K$ form a basis for an algebra, the
\emph{descent algebra} of $W$.

Independently, interpretations of this theorem involving
certain matrices  have been developed for the descent
algebras of the Coxeter groups of types $A$ and $B$ as
a means of obtaining further results about these
algebras,
\cite{garsia-reutenauer},\cite{garsia-remmel},\cite{atkinson-solomon}, \cite{bergeron-bergeron},
\cite{bergeron-n}. However, no such matrix interpretation was known for the Coxeter groups of type
$D$. In this paper we shall develop some tools and a lemma from which we can easily deduce the
matrix interpretation of Theorem ~\ref{solomon-theorem} for the descent algebra of the Coxeter
groups of type $A$. The purpose of this paper, however, is not just to give another proof of a
well known result, instead it is to act as a precursor to a subsequent paper in which we
formulate the missing matrix interpretation for the descent algebras of the Coxeter groups of
type $D$, \cite{bergeron-vanwilli}.

To develop our tools, let us take our Coxeter group
$W$ to be the   Coxeter group of type
$A$ with $n-1$ fundamental reflections, that is the symmetric group $S_n$. More specifically,  let
us take $S_n$ to be the group of permutations acting on  the set 
$N=\{1,\ldots ,n\}$, with generating set
$S$,  where $S$ is the set of
$n-1$ transpositions  $s_1,s_2,\ldots ,s_{n-1}$, such that $s_i=(i,i+1)$\index{generating
set!symmetric groups}.
 
If $J$ is a subset of $S$, then we define the \emph{graph} \index{graph!symmetric groups}$\cJ =(N
,\cE )$ of $J$ to be the graph with vertex set $\{ 1,\ldots ,n\}$, and edge set $\cE =\{ (i,i+1)\in
J\}$. In general we shall use roman capitals $J,K,\ldots$ for subsets of $S$, and their
calligraphic counterparts $\cJ ,\cK ,\ldots$ for their associated graphs. Suppose now that $\cJ$
has $r$ connected components. A   set of vertices is associated with each component, and we can
order these sets by their least elements in a natural way. Once ordered, we can label them   ${\cal
J}_1,\ldots ,{\cal J}_r$ such that  $1\in {\cal J}_1$, and define the\emph{ ordered presentation}
\index{ordered presentation!symmetric groups|emph}of $\cJ$ to be the ordered list
\[(\cJ _1,\ldots ,\cJ _r).\]
Note that this is the canonical ordered set partition associated to $J$, and
 that if $u\in\cJ _i$ and $v\in \cJ _j$, and $i<j$, then $u<v$.

\begin{example}
In $S_9$. If $J=\{(2,3),(3,4),(7,8)\}$, then $\cJ$ is
\begin{figure}[htbp]
\centerline{\epsffile{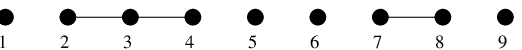}}
\end{figure}
 
with

\[\cJ_1=\{1\}, \cJ_2=\{2,3,4\}, \cJ_3=\{5\}, \cJ_4=\{6\}, \cJ_5=\{7,8\}, \cJ_6=\{9\} .\]
The ordered presentation of $\cJ$ is  \[(\{1\},\{2,3,4\},\{5\},\{6\},\{7,8\},\{9\}).\]
\end{example}

Let us define  $\cW _{\cJ _i}$ to be the subgroup of $S_n$ that consists of all permutations of
$S_n$ that   fix all points outside of
$\cJ _i$, and let
\[\cW _{\cJ}=\cW _{\cJ _1}\times\ldots\times\cW _{\cJ _r}.\]
Observe that $\cW _\cJ =W_J$. If we let $\kappa$ be a composition of $n$, with components $\kappa
_1, \kappa _2, \ldots ,\kappa _r$, and define
\[\boldsymbol{S_\kappa}=S_{\kappa _1}\times\ldots\times S_{\kappa _r},\]
then 
$\boldsymbol{S_\kappa}\cong\cW _\cJ$, where the    sets
$\cJ _i$ of $\cJ$ satisfy
\begin{eqnarray}&|\cJ _i|=\kappa _i.&\label{scor-respond}
\index{composition!algorithm for!symmetric
groups}\end{eqnarray}

Let us now take $J$ and $K$ to be any subsets of $S$,  and $x\in W$, and 
let $x\cJ$ denote the image of the graph $\cJ$ under $x$; that is $(x(i),x(j))$ is an edge in $x\cJ
$ if and only if $(i,j)$ is an edge in $\cJ$. Let $\cJ\cap\cK$ be the graph with vertex set $N$
whose edges are those present in both $\cJ$ and $\cK$. The ordered presentation of
$x^{-1}\cJ\cap\cK$, where $x\in X_J^{-1}\cap X_K$, is given by the following lemma.

\begin{lemma}\index{ordered presentation!symmetric groups}
Let $J$ and $K$ be subsets of $S$, and let  $x\in X_J^{-1}
\cap X_K$. Let the ordered  presentation of $\cJ$ be $(\cJ
_1,\ldots ,\cJ _r)$, and $\cK$ be $(\cK _1,\ldots ,\cK
_s)$. Then the ordered
 presentation of  
$x^{-1}\cJ\cap\cK$ is 
 \begin{eqnarray}\nonumber&(x^{-1}\cJ_1\cap \cK_1,
x^{-1}\cJ_2\cap \cK_1, \ldots ,x^{-1}\cJ _r\cap\cK
_1,&\\\nonumber 
& x^{-1}\cJ _1\cap\cK _2, x^{-1}\cJ _2\cap\cK
_2,\ldots ,x^{-1}\cJ _r\cap\cK
_2,&\\\label{j-andk}
 & \ldots ,&\\\nonumber 
&x^{-1}\cJ _1\cap\cK _s, x^{-1}\cJ _2\cap\cK _s ,\ldots
x^{-1}\cJ_r\cap
\cK_s)&\\\nonumber\end{eqnarray}with empty
sets removed.\label{sord-ered}
\end{lemma}

\begin{proof}
To prove that
(\!~\ref{j-andk}) 
is the ordered presentation of 
$x^{-1}\cJ\cap\cK$, it is sufficient to prove the
following two statements

\begin{enumerate}
\item The elements of each set are less than those that appear in any set later in the list.
\item Each non-empty set    $x^{-1}\cJ _q\cap\cK _m$ is indeed the vertex set of a connected component
of $x^{-1}\cJ\cap\cK$.
\end{enumerate}
To prove statement $1$ we must show that
 
$1.1$ The elements in $x^{-1}\cJ _r\cap\cK _m$ are less than those in $x^{-1}\cJ _1\cap\cK _{m+1}$

$1.2$ The elements in $x^{-1}\cJ _q\cap\cK _m$ are less than those in $x^{-1}\cJ _{q+1}\cap\cK
_{m}$.
 
Case $1.1$ follows immediately, since all vertices in  $\cK _m$ are less than
those in $\cK _{m+1}$ by definition.

To prove case $1.2$ we need only show that if the vertex  $i\in
x^{-1}\cJ_q\cap \cK_m$ and $j\in x^{-1}\cJ_{q+1}\cap \cK_m$, then $i<j$. To do this we shall first prove that 
all  vertices in $\cK _m$ appear from left to right in increasing order in  the list
$x^{-1}(1),x^{-1}(2),\ldots ,x^{-1}(n)$. 
 
From the definition of $X_K$ as a set of minimal length coset representatives, it follows that if 
$x\in X_K$ then $l(xk)>l(x)$ for all $k\in K$.  In $S_n$, $l(x)$ \index{length of a word!symmetric
groups}is the number of inversions in $x$, that is the number of $h<l$ for which
$x(l)<x(h)$, \cite{humphrey-coxeter}. Hence it follows
that for all
$k=(h,h+1)\in K$ we have
$x(h)<x(h+1)$, since $xk$, $x $ differ only in the reversing of $h$ and
$h+1$. From this we can deduce that  $h$ is to the left of $h+1$  in the list 
 \[x^{-1}(1),x^{-1}(2),\ldots ,x^{-1}(n).\]
  Now suppose that $i\in  x^{-1}\cJ _q\cap \cK _m$,    $j\in x^{-1}\cJ_{q+1}\cap \cK_m$. Then  
$x(i)=u\in \cJ_q$, and
$ x(j)=v\in \cJ_{q+1}$. It follows that    $u<v$, and so $x^{-1}(u)$ appears before $x^{-1}(v)$
in 
$x^{-1}(1),x^{-1}(2),\ldots ,x^{-1}(n)$. However, $x^{-1}(u)=i$ and $x^{-1}(v)=j$, and  since we know that the vertices of
$\cK _m$ appear in increasing order from left to right in  $x^{-1}(1),x^{-1}(2),\ldots ,x^{-1}(n)$, 
it follows that
$i<j$.

Statement $2$ will follow if we can prove the following assertions.
 
$2.1$ The sets $x^{-1}\cJ _q\cap\cK _m$ are all disjoint.

$2.2$ No edge in $x^{-1}\cJ\cap \cK$ connects vertices in different subsets  $x^{-1}\cJ _q\cap\cK
_m$ and $x^{-1}\cJ _{q'}\cap\cK _{m'}$. 

$2.3$ For every $i,i+1\in x{\cal J}_q\cap{\cal K}_m$,  an
edge exists   in
$x^{-1}\cJ\cap \cK$ between $i$, $i+1$.

Again,   assertion $2.1$ follows since all $\cJ _q$ and $\cK _m$ are disjoint and $x$ is a 
bijection from $N$ to itself.

To prove   assertion $2.2$, let $(u,v)$ be an edge in $x^{-1}\cJ\cap \cK$, such that
 $u\in x^{-1}\cJ _q\cap\cK _m$ and $v\in x^{-1}\cJ _{q'}\cap\cK _{m'}$. We know that $\cJ _q$ and $\cJ _{q'}$
are vertex sets of connected components of $\cJ$, so $x^{-1}\cJ _q$ and $x^{-1}\cJ _{q'}$ must be vertex sets
of connected components of $x^{-1}\cJ$. Hence, $q=q'$. Similarly, $\cK _m$ and $\cK _{m'}$ are vertex
sets of connected components of $\cK$, and so $m=m'$.

For   assertion $2.3$, let  $i,i+1\in x^{-1}\cJ _q\cap\cK _m$ and let $ x(i)=u$, and $
x(i+1)=u+l$. Since we know from the proof of case $1.2$ that all $i\in {\cal K}_m$
appear in increasing order from left to right in the list $x^{-1}(1),x^{-1}(2),\ldots ,x^{-1}(n)$,
we can deduce that  $l\geq 1$. We can also deduce that because $
X_J^{-1}$ is defined as a set of minimal length right coset representatives, we have that
$x^{-1}(v)<x^{-1}(v+1)$ for all $(v,v+1)\in J$.

Therefore, since $u, u+l\in {\cal J}_q$, we have that $u+k\in {\cal J}_q$  for all $k=0,\ldots ,l$
such that
\[x^{-1}(u)<x^{-1}(u+1)<\ldots <x^{-1}(u+l-1)<x^{-1}(u+l).\]
However,   $x^{-1}(u)=i$, $x^{-1}(u+l)=i+1$, so it follows that $l=1$, and  so, by definition  
$(u,u+l)$ is an edge in  $\cJ$. Therefore, since $x^{-1}(u)=i$ and $x^{-1}(u+l)=i+1$, it follows
that $(i,i+1)$ is an edge in 
$x^{-1}\cJ\cap\cK$, and we are done. \end{proof}

As a consequence of  Lemma 2 \cite{solomon-mackey}, and our
Lemma ~\ref{sord-ered} \begin{eqnarray*} x^{-1}W_Jx\cap
W_K&=&W_{x^{-1}Jx\cap K}\\ &=&\cW _{x^{-1}\cJ\cap\cK}\\
&=&\cW_{x^{-1}\cJ_1\cap \cK_1}\times\ldots\times \cW_{x^{-1}\cJ_r\cap \cK_s}\\ 
&=&(x^{-1}\cW_{\cJ_1}x\cap \cW_{\cK_1})\times\ldots\times (x^{-1}\cW_{\cJ_r}x\cap
\cW_{\cK_s})\\
&\cong&(x^{-1}S_{\kappa _1}x\cap S_{\nu _1})\times\ldots\times
(x^{-1}S_{\kappa _r}x\cap S_{\nu _s})\\
\end{eqnarray*}
where $\kappa$ and $\nu$ are suitable compositions of $n$ determined by $J$, $K$ respectively, according to condition
(\!~\ref{scor-respond}).
Note that the final isomorphism symbol is an equality if $x^{-1}S_{\kappa _i}x\cap S_{\nu _j}$ is regarded as the group of
permutations on $x^{-1}\cJ_{i}\cap \cK_{j}.$

Let
\[z_{ij}=|x^{-1}\cJ_{i}\cap \cK_{j}|\]then, by Theorem 1.3.10 \cite{james-kerber}, 
we have a bijective mapping
\[\zeta :x\mapsto (z_{ij})\]
from $X_{J}^{-1}\cap X_K$ into the set of $s\times r$ matrices with non-negative  integer entries,
$\boldsymbol{z}=(z_{ij}) $, which satisfy \[\sum _i z_{ij}=\kappa _j,\ \sum _j z_{ij}=\nu _i.\]
Observe that reading the non-zero   entries of the matrix $\boldsymbol{z}$ by row give a 
composition, $\eta$,
of $n$. We say that $\eta$ is the\emph{ reading word} \index{reading word!symmetric groups} of 
$\boldsymbol{z}$, and note that $\boldsymbol{S_\eta}$ is isomorphic to $W_{x^{-1}Jx\cap K}$. We
also observe that each matrix corresponds to one $x\in X_J^{-1}\cap X_K$, given in Solomon's
Theorem. Therefore, if we now rename the basis elements such that ${\cal X}_J$
 becomes $B_\kappa$, where the components of $\kappa$ in order are the sizes of the vertex sets of
$\cJ$ taken in the natural order,  we can recast
 Solomon's Theorem in terms of compositions and matrices as follows.

 \begin{thm} For every composition $\nu$   of $n$, let $X_{\nu}$ be the unique set of
minimal length left coset representatives of $S_n/\boldsymbol{S_{\nu}}$. Let 
\[B_{\nu}=\sum_{\sigma\in X_{\nu}}\sigma.\]

If $\kappa, \nu$ are compositions of $n$, then
 \[B_{\kappa}B_{\nu}=\sum _{\boldsymbol{z}} B_{\eta}\]
where the sum is over all matrices $\boldsymbol{z}=(z_{ij})$ with non-negative integer entries that satisfy
\begin{enumerate}
\item $\sum _i z_{ij}=\kappa _j$,
\item $\sum _j z_{ij}=\nu _i$.\end{enumerate}
For each matrix, $\boldsymbol{z}$,  $\eta$ is the reading word of $\boldsymbol{z}$.

\end{thm}
This is precisely the classical matrix interpretation of Solomon's Theorem for the symmetric groups, for instance Proposition 1.1, \cite{garsia-reutenauer}.

\begin{ack}
The author is indebted to Michael Atkinson for  many useful discussions, and grateful to Nantel Bergeron for critical comments.
\end{ack}

 \end{document}